\documentclass{amsart}
\usepackage{amsmath,amssymb,amsthm,amscd,amsfonts}
\usepackage{enumerate,url,tikz,centernot}
\thanks{This work was partially supported by NSF grant DMS-1115455}

\newcommand{\C}{\mathbb{C}}

\newcommand{\Z}{\mathbb{Z}}

\newcommand{\F}{\mathbb{F}}
\newcommand{\kk}{\mathbb{F}}
\newcommand{\kbar}{\overline\kk}
\newcommand{\Fp}{\F_p}
\newcommand{\Fq}{\F_q}

\newcommand{\fa}{\mathfrak{a}}

\newcommand{\inkron}[2]{\genfrac {(}{)}{0.9pt}{}{#1}{#2}}
\newcommand{\M}{\textsf{M}}
\newcommand{\Ell}{{\rm Ell}}
\def\llog{\operatorname{llog}}
\def\lllog{\operatorname{lllog}}
\newcommand{\logpow}[2]{{\hspace{.7pt}\log\hspace{-1.3pt}{}^{#1}\hspace{-0.4pt}#2}}
\newcommand{\llogpow}[2]{{\hspace{.7pt}\llog\hspace{-1.3pt}{}^{#1}\hspace{-0.4pt}#2}}
\newcommand{\lllogpow}[2]{{\hspace{.7pt}\lllog\hspace{-1.3pt}{}^{#1}\hspace{-0.4pt}#2}}
\renewcommand{\vec}[1]{\boldsymbol{#1}}
\renewcommand{\O}{\mathcal{O}}
\def\disc{\operatorname{disc}}
\def\cl{\operatorname{cl}}
\def\Gal{\operatorname{Gal}}

\def\End{\operatorname{End}}

\def\ndiv{\centernot\mid}
\newcommand{\jt}{\tilde\jmath}
\newcommand{\jh}{\hat\jmath}
\newcommand{\gt}{\tilde{g}}
\newtheorem{theorem}{Theorem}

\newtheorem{proposition}[theorem]{Proposition}
\newtheorem{definition}[theorem]{Definition}
\newtheorem{remark}[theorem]{Remark}
\newtheorem{heuristic}{Heuristic}

\begin{document}

\title{On the evaluation of modular polynomials}
\author{Andrew V. Sutherland}
\address{Department of Mathematics\\Massachusetts Institute of Technology\\Cambridge, Massachusetts\ \  02139}
\email{drew@math.mit.edu}
\begin{abstract}
We present two algorithms that, given a prime $\ell$ and an elliptic curve $E/\Fq$,
directly compute the polynomial $\Phi_\ell(j(E),Y)\in\Fq[Y]$ whose roots are the
$j$-invariants of the elliptic curves that are $\ell$-isogenous to~$E$.
We do not assume that the modular polynomial $\Phi_\ell(X,Y)$ is given.
The algorithms may be adapted to handle other types of modular polynomials, and we
consider applications to point counting and the computation of endomorphism rings.
We demonstrate the practical efficiency of the algorithms by setting a new
point-counting record, modulo a prime $q$ with more than 5,000 decimal digits,
and by evaluating a modular polynomial of level $\ell =$ 100,019.
\end{abstract}

\maketitle

\section{Introduction}
Isogenies play a crucial role in the theory and application of elliptic curves.
A standard method for identifying (and computing) isogenies uses the classical modular polynomial $\Phi_\ell\in\Z[X,Y]$, which parameterizes pairs of $\ell$-isogenous elliptic curves in terms of their $j$-invariants.  More precisely,
over a field $\kk$ of characteristic not equal to $\ell$, the modular equation
\[
\Phi_\ell\bigl(j_1,j_2\bigr)=0
\]
holds if and only if $j_1$ and $j_2$ are the $j$-invariants of elliptic curves defined over~$\kk$ that are related by a cyclic isogeny of degree $\ell$.
In practical applications, $\kk$ is typically a finite field $\Fq$, and $\ell$ is a prime, as we shall assume throughout.
For the sake of simplicity we assume that $q$ is prime, but this is not essential.

A typical scenario is the following: we are given an elliptic curve $E/\Fq$ and wish to determine whether $E$ admits an $\ell$-isogeny defined over $\Fq$, and if so, to identify one or all of the elliptic curves that are $\ell$-isogenous to $E$.  This can be achieved by computing the instantiated modular polynomial
\[
\phi_\ell(Y)=\Phi_\ell(j(E),Y) \in \Fq[Y],
\]
and finding its roots in $\Fq$ (if any).
Each root is the $j$-invariant of an elliptic curve that is $\ell$-isogenous to $E$ over $\Fq$, and every such $j$-invariant is a root of $\phi_\ell(Y)$.

For large $\ell$ the main obstacle to obtaining $\phi_\ell$ is the size of $\Phi_\ell$, which is $O(\ell^3\log\ell)$ bits; several gigabytes for $\ell\approx 10^3$, and many terabytes for $\ell\approx 10^4$, see \cite[Table~1]{BrokerLauterSutherland:CRTModPoly}.
In practice, alternative modular polynomials that are smaller than $\Phi_\ell$ by a large constant factor are often used, but their size grows at the same rate and this quickly becomes the limiting factor, as noted in \cite[\S5.2]{Enge:ModularPolynomials} and elsewhere.
The following quote is taken from the 2009 INRIA Project-Team TANC report \cite[p.~9]{TANC:PointCounting}:
\smallskip

\emph{``\ldots computing modular polynomials remains the stumbling block for new point counting records.
Clearly, to circumvent the memory problems, one would need an algorithm that directly obtains the polynomial specialized in one variable."}
\smallskip

Here we present just such an algorithm (two in fact), based on the isogeny volcano approach of \cite{BrokerLauterSutherland:CRTModPoly}.
Our basic strategy is to compute the instantiated modular polynomial $\phi(Y)=\Phi_\ell(j(E),Y)$ modulo many ``suitable" primes $p$ and apply the explicit Chinese remainder theorem modulo $q$ (see \S\ref{subsection:crt} and \S\ref{subsection:modpoly} for a discussion of the explicit CRT and suitable primes).
However, two key issues arise.

First, if we simply lift the $j$-invariant $j(E)$ from $\Fq\simeq\Z/q\Z$ to $\Z$ and reduce the result modulo $p$, when we instantiate $\Phi_\ell(j(E),Y)$ the powers of $j(E)$ we compute may correspond to integers that are much larger than the coefficients of $\Phi_\ell$, forcing us to use many more CRT primes than we would otherwise need.
We address this issue by instead powering in $\Fq$, lifting the powers to $\Z$, and then reducing them modulo $p$.
This yields our first algorithm, which is well-suited to situations where~$q$ is much larger than $\ell$, say $\log q \approx \ell$, as in point-counting applications.

Second, to achieve the optimal space complexity, we must avoid computing $\Phi_\ell\bmod p$.
Indeed, if $\log q\approx \log \ell$, then $\Phi_\ell\bmod p$ will not be much smaller than $\Phi_\ell\bmod q$.
Our second algorithm uses an online approach to avoid storing all the coefficients of $\Phi_\ell\bmod p$ simultaneously.  This algorithm is well-suited to situations where $\log q$ is not dramatically larger than $\log \ell$, say $O(\log\ell)$ or $O(\logpow{2}{\ell})$.
This occurs, for example, in algorithms that compute the endomorphism ring of an elliptic curve \cite{BissonSutherland:Endomorphism},
or algorithms to evaluate isogenies of large degree \cite{Jao:LargeIsogenies}.

Under the generalized Riemann hypothesis (GRH), our first algorithm has an expected running time of $O(\ell^3\logpow{3}{\ell}\llog\ell)$ and uses $O(\ell^2\log\ell + \ell\log q)$ space, assuming $\log q = O(\ell\log\ell)$.\footnote{See Theorem~\ref{thm:alg1comp} for a more precise bound.  We write $\llog n$ for $\log\log n$ throughout.}
This time complexity is the same as (and in practice faster than) the time to compute $\Phi_\ell$, and the space complexity is reduced by up to a factor of $\ell$.
When $\log q\approx\ell$ the space complexity is nearly optimal, quasi-linear in the size of $\phi_\ell$.
The second algorithm uses $O(\ell^3(\log q+\log \ell)\logpow{1+o(1)}\ell)$ time and $O(\ell\log q+\ell\log\ell)$ space, under the GRH.
Its space complexity is optimal for $q=\Omega(\ell)$, and when $\log q = O(\logpow{2-\epsilon}{\ell})$ its time complexity is better than the time to compute $\Phi_\ell$.
For $\log q\gg \log^2\ell$ its running time becomes less attractive and the first algorithm may be preferred, or see \S\ref{subsection:hybrid} for a hybrid approach.

In conjunction with the SEA algorithm, the first algorithm allows us to compute the cardinality of an elliptic curve modulo a prime $q$ with a heuristic\footnote{The heuristic relates to the distribution of Elkies primes and is a standard assumption made when using the SEA algorithm (without it there is no advantage over Schoof's algorithm).}
 running time of $O(n^4\logpow{3}{n}\llog n)$, using $O(n^2\log n)$ space, where $n = \log q$.  To our knowledge, all alternative approaches applicable to prime fields increase at least one of these bounds by a factor of $n$ or more.
The running time is competitive with SEA implementations that rely on precomputed modular polynomials (as can be found in Magma \cite{Magma} and PARI \cite{PARI}), and can easily handle much larger values of $q$.

As an important practical optimization, we also evaluate modular polynomials $\phi_\ell^f(Y)= \Phi_\ell^f(f(E),Y)$ defined by modular functions $f(z)$ other than the $j$-function.  This includes the Weber $\frak{f}$-function, whose modular polynomials are smaller than the classical modular polynomial by a factor of 1728 and can be computed much more quickly (by roughly the same factor).  This speedup also applies to $\phi_\ell^f$.

To demonstrate the capability of the new algorithms, we use a modified version of the SEA algorithm to count points on an elliptic curve
modulo a prime of more than 5,000 decimal digits, and evaluate a modular polynomial of level $\ell=$ 100,019 modulo a prime of more than 25,000 decimal digits.
\subsection*{Acknowledgements} I am grateful to David Harvey for his assistance with the algorithms for fast polynomial arithmetic used in the computations described in~\S\ref{section:computations}, and to Daniel Kane for suggesting the hybrid approach outlined in \S\ref{subsection:hybrid}.
I also thank the referees for their comments and helpful suggestions.

\section{Background}
This section contains a brief summary of background material that can be found in standard references such as \cite{Lang:EllipticFunctions,Silverman:EllipticCurves1,Silverman:EllipticCurves2}, or in the papers \cite{BrokerLauterSutherland:CRTModPoly,Sutherland:HilbertClassPolynomials}, both of which exploit isogeny volcanoes using a CRT-based approach, as we do here.
For the sake of brevity, we recall only the results we need, and only in the generality necessary.

To simplify the presentation, we assume throughout that $\Fp$ and $\Fq$ denote prime fields with $\ell\ne p,q$, and, where relevant, that $q$ is sufficiently large (typically $q>2\ell$).  But this assumption is not needed for our main result; Algorithms~1 and~2 work correctly for any prime $q$ (even $q=\ell$), and can be extended to handle non-prime $q$.

\subsection{Isogenies}\label{subsection:isogenies}
Let $E$ be an elliptic curve defined over a field $\kk$.
Recall that an isogeny $\psi:E\to\tilde E$ is a morphism of elliptic curves that is also a group homomorphism from $E(\kbar)$ to $\tilde E(\kbar)$.
The kernel of a nonzero isogeny is a finite subgroup of~$E(\kbar)$, and when $\psi$ is separable, the size of its kernel is equal to its degree. 
Conversely, every finite subgroup $G$ of $E(\kbar)$ is the kernel of a separable isogeny (defined over the fixed field of the stabilizer of $G$ in  $\Gal(\kbar/\kk)$).
We say that $\psi$ is \emph{cyclic} if its kernel is cyclic, and call $\psi$ an $N$-\emph{isogeny} when it has degree $N$.
Note that an isogeny of prime degree $\ell\ne{\rm char}(\kk)$ is necessarily cyclic and separable.

The classical modular polynomial $\Phi_N$ is the minimal polynomial of the function $j(N z)$ over the field $\C(j)$, where $j(z)$ is the modular $j$-function.
As a polynomial in two variables, $\Phi_N\in\Z[X,Y]$ is symmetric in $X$ and $Y$ and
has the defining property that the roots of $\Phi_\ell(j(E),Y)$ are precisely the $j$-invariants of the elliptic curves $\tilde E$ that are related to $E$ by a cyclic $N$-isogeny.  In this paper $N=\ell$ is prime, in which case $\Phi_\ell(X,Y)$ has degree $\ell+1$ in each variable.

If $E$ is given by a short Weierstrass equation $Y^2=X^3+a_4X + a_6$, then $\psi$ can be expressed in the form $\psi(x,y)=(\psi_1(x),cy\frac{d}{dx}\psi_1(x))$ for some $c\in\kbar^*$.
When $c=1$ we say that $\psi$ and its image are \emph{normalized}.
Given a finite subgroup $G$ of $E(\kbar)$, a normalized isogeny with $G$ as its kernel can be constructed using V\'elu's formulae~\cite{Velu:Isogenies}, along with an explicit equation for its image $\tilde E$.
Conversely, suppose we are given a root $\jt=j(\tilde E)$ of $\phi_\ell(Y)=\Phi_\ell(j(E),Y)$, and also the values of $\Phi_X(j,\jt)$, $\Phi_Y(j,\jt)$, $\Phi_{XX}(j,\jt)$, $\Phi_{XY}(j,\jt)$, and $\Phi_{YY}(j,\jt)$, where $j=j(E)$ and
\[
\textstyle{\Phi_X=\frac{\partial}{\partial X}\Phi_\ell,\medspace\medspace
\Phi_Y=\frac{\partial}{\partial Y}\Phi_\ell,\medspace\medspace
\Phi_{XX}=\frac{\partial^2}{\partial X^2}\Phi_\ell,\medspace\medspace
\Phi_{XY}=\frac{\partial^2}{\partial X\partial Y}\Phi_\ell,\medspace\medspace
\Phi_{YY}=\frac{\partial^2}{\partial Y^2}\Phi_\ell.}
\]
To this data we may apply an algorithm of Elkies \cite{Elkies:AtkinBirthday} that computes an equation for $\tilde E$ that is the image of a normalized $\ell$-isogeny $\psi\colon E\to\tilde E$, along with an explicit description of its kernel: the monic polynomial $h_\ell(X)$ whose roots are the abcissae of the non-trivial points in $\ker \psi$; see \cite[Alg.\ 27]{Galbraith:MathPKC}.
The quantities $\Phi_{XX}(j,\jt)$, $\Phi_{XY}(j,\jt)$, and $\Phi_{YY}(j,\jt)$ are not strictly necessary; the equation for $\tilde E$ depends only on $j$, $\jt$, $\Phi_X(j,\jt)$ and $\Phi_Y(j,\jt)$, and we may then apply algorithms of Bostan et al.~\cite{Bostan:FastIsogenies} to compute $h_\ell(X)$ (and an equation for $\psi$) directly from $E$ and $\tilde E$.

\subsection{Explicit CM theory}\label{subsection:cm}

Recall that the endomorphism ring of an ordinary elliptic curve $E$ over a finite field $\Fp$ is isomorphic to an order $\O$ in an imaginary quadratic field $K$.  In this situation $E$ is said to have \emph{complex multiplication} (CM) by $\O$.
The elliptic curve $E/\Fp$ is the reduction of an elliptic curve $\hat E/\C$ that also has CM by $\O$.  The $j$-invariant of $\hat E$ generates the ring class field $K_\O$ of $\O$, and its minimal polynomial over $K$ is the \emph{Hilbert class polynomial} $H_\O\in\Z[X]$, whose degree is the class number $h(\O)$.\footnote{As in \cite{Belding:HilbertClassPolynomial}, we call $H_\O$ a Hilbert class polynomial even when $\O$ is not the maximal order.}
The prime $p$ splits completely in $K_\O$, and~$H_\O$ splits completely in $\Fp[X]$.
For $p>3$, the prime $p$ splits completely in $K_\O$ if and only if it satisfies the norm equation $4p=t^2-v^2D$, where $D=\disc(\O)$,
and for $D<-4$ the integers $t=t(p)$ and $v=v(p)$ are uniquely determined up to sign.

We define the set
\[
\Ell_\O(\Fp)=\{j(E): E/\Fp \text{ with } \End(E)\simeq\O\},
\]
which consists of the roots of $H_\O$ in $\Fp$.
Let $\iota\colon\O\hookrightarrow\End(E)$ denote the normalized embedding (so $\iota(\alpha)^*\omega=\alpha\omega$ for all $\alpha\in\O$ and invariant differentials $\omega$ on~$E$; c.f. \cite[Prop.\ II.1.1]{Silverman:EllipticCurves2}).
The ideals of $\O$ act on $\Ell_\O(\Fp)$ via isogenies as follows.
Let $\fa$ be an $\O$-ideal of norm $N$, and define $E[\fa]=\cap_{\alpha\in\fa}\ker\iota(\alpha)$.
There is a separable $N$-isogeny from $E$ to $\tilde{E}=E/E[\fa]$, and the action of $\fa$ sends $j(E)$ to $j(\tilde{E})$.
Principal ideals act trivially, and this induces a regular action of the class group $\cl(\O)$ on $\Ell_O(\Fp)$.
Thus $\Ell_O(\Fp)$ is a principal homogeneous space, a \emph{torsor}, for $\cl(\O)$.

Writing the $\cl(\O)$-action on the left, we note that if $\fa$ has prime norm $\ell$, then $\Phi_\ell(j,[\fa]j)=0$ for all $j\in\Ell_\O(\Fp)$.
For $\ell$ not dividing $v(p)$, the polynomial $\phi_\ell(Y)=\Phi_\ell(j,Y)$ has either one or two roots in $\Fp$, depending on whether $\ell$ ramifies or splits in $K$.
In the latter case, the two roots $[\fa]j$ and $[\fa^{-1}]j$ can be distinguished using the Elkies kernel polynomial $h_\ell(X)$, as described in \cite[\S 5]{Broker:pAdicClassPolynomial} and \cite[\S 3]{Galbraith:GHSattack}.

\subsection{Polycyclic presentations}\label{subsection:pcp}
In order to efficiently realize the action of $\cl(\O)$ on $\Ell_\O(\Fp)$, it is essential to represent elements of $\cl(\O)$ in terms of a set of generators with small norm.
We will choose $\O$ so that $\cl(\O)$ is generated by ideals of norm bounded by $O(1)$, via \cite[Thm.\ 3.3]{BrokerLauterSutherland:CRTModPoly}, but these generators will typically not be independent.
Thus as explained in \cite[\S 5.3]{Sutherland:HilbertClassPolynomials}, we use polycyclic presentations.

Any sequence of generators $\vec\alpha=(\alpha_1,\ldots,\alpha_k)$ for a finite abelian group $G$ defines a polycyclic series
\[
1=G_0\lhd G_1\lhd\ldots\lhd G_{k-1}\lhd G_k =G,
\]
with $G_i=\langle\alpha_1,\ldots,\alpha_i\rangle$, in which every quotient $G_i/G_{i-1}\simeq\langle \alpha_i\rangle$ is necessarily cyclic.
We associate to $\vec\alpha$ the sequence of \emph{relative orders} $r(\vec\alpha)=(r_1,\ldots,r_k)$ defined by $r_i=|G_i:G_{i-1}|$.
Every element $\beta\in G$ has a unique $\vec\alpha$-\emph{representation} of the form
\[
\beta = \vec{\alpha}^{\vec{e}}=\alpha_1^{e_1}\cdots\alpha_k^{e_k}\qquad (0\le e_i < r_i).
\]
We also associate to $\vec\alpha$ the matrix of power relations $s(\vec\alpha)=[s_{ij}]$ defined by
\[
\alpha_i^{r_i} = \alpha_1^{s_{i,1}}\alpha_2^{s_{i,2}}\cdots \alpha_{i-1}^{s_{i,i-1}}\qquad (0\le s_{ij}<r_j),
\]
with $s_{ij}=0$ for $i\le j$.

We call $\vec\alpha$, together with $r(\vec\alpha)$ and $s(\vec\alpha)$, a \emph{(polycyclic) presentation} for $G$,
and if all the $r_i$ are greater than 1, we say that the presentation is \emph{minimal}.
A generic algorithm to compute a minimal polycyclic presentation is given in \cite[Alg.\ 2.2]{Sutherland:HilbertClassPolynomials}.
Having constructed such an $\vec\alpha$, we can efficiently enumerate $G=\cl(\O)$ (or the torsor $\Ell_\O(\Fq)$, given a starting point), by enumerating $\vec{\alpha}$-representations.

\subsection{Explicit CRT}\label{subsection:crt}
Let $p_1,\ldots,p_n$ be primes with product $M$, let $M_i=M/p_i$, and let $a_iM_i\equiv 1 \bmod p_i$.
If $c\in\Z$ satisfies $c\equiv c_i\bmod p_i$, then $c\equiv\sum_i c_ia_iM_i\bmod M$.
If $M > 2|c|$, this congruence uniquely determines~$c$.
This is the usual CRT method.

Now suppose $M > 4|c|$ and let $q$ be a prime (or any integer).
Then we may apply the \emph{explicit CRT mod $q$} \cite[Thm.\ 3.1]{Bernstein:ModularExponentiation} to compute
\begin{equation}\label{eq:ecrt}
c\equiv \Bigl(\sum_i c_ia_iM_i - rM\Bigr)\bmod q,
\end{equation}
where $r$ is the closest integer to $\sum_i c_ia_i/p_i$; when computing $r$, it suffices to approximate each $c_ia_i/p_i$ to within $1/(4n)$, by \cite[Thm.\ 2.2]{Bernstein:ModularExponentiation}.

As described in \cite[\S 6]{Sutherland:HilbertClassPolynomials}, we may use the explicit CRT to simultaneously compute $c\bmod q$ for many integers $c$ (the coefficients of $\phi_\ell$, for example), using an \emph{online algorithm}.
We first precompute the $a_i$ and $a_iM_i \bmod q$.
Then, for each prime~$p_i$, we determine the values $c_i$ for all the coefficients $c$ (by computing $\phi_\ell\bmod p_i$), update two partial sums for each coefficient, one for $\sum c_ia_iM_i\bmod q$ and one for $\sum c_ia_i/p_i$, and then discard the $c_i$'s.
When the computations for all the~$p_i$ have been completed (these may be performed in parallel), we compute $r$ and apply \eqref{eq:ecrt} for each coefficient.
The space required by the partial sums is just $O(\log q)$ bits per coefficient.  
See \cite[\S 6]{Sutherland:HilbertClassPolynomials} for further details, including algorithms for each step.

\subsection{Modular polynomials via isogeny volcanoes}\label{subsection:modpoly}
For distinct primes $\ell$ and~$p$, we define the \emph{graph of $\ell$-isogenies} $\Gamma_\ell(\Fp)$, with vertex set $\Fp$ and edges $(j_1,j_2)$ present if and only if $\Phi_\ell(j_1,j_2)=0$.
Ignoring the connected components of $0$ and $1728$, the ordinary components of $\Gamma_\ell(\Fp)$ are $\ell$-\emph{volcanoes} \cite{Fouquet:IsogenyVolcanoes,Kohel:thesis}, a term we take to include cycles as a special case \cite{Sutherland:HilbertClassPolynomials}.
In this paper we focus on $\ell$-volcanoes of a particular form, for which we can compute $\Phi_\ell\bmod p$ very quickly, via \cite[Alg.~2.1]{BrokerLauterSutherland:CRTModPoly}.

Let $\O$ be an order in an imaginary quadratic field $K$ with maximal order $\O_K$, let $\ell$ be an odd prime not dividing $[\O_K:\O]$, and assume $D=\disc(\O)<-4$.
Let $p$ be a prime of the form $4p=t^2-\ell^2v^2D$ with $\ell\ndiv v$ and $p\equiv 1\bmod \ell$.
Then $p$ splits completely in the ring class field of $\O$, but not in the ring class field of the order with index $\ell^2$ in $\O$.
The requirement $p\equiv 1\bmod\ell$ ensures that for $j(E)\in\Ell_\O(\Fp)$ we can choose $E$ so that $E[\ell]\subset E(\Fp)$, which is critical to the efficiency of both the algorithm in \cite{BrokerLauterSutherland:CRTModPoly} and our algorithms here.

The components of $\Gamma_\ell(\Fp)$ that intersect $\Ell_\O(\Fp)$ are isomorphic $\ell$-volcanoes with two levels: the \emph{surface}, whose vertices lie in $\Ell_\O(\Fp)$, and the \emph{floor}, whose vertices lie in $\Ell_{\O'}(\Fp)$, where $\O'$ is the order of index $\ell$ in $\O$.
Each vertex on the surface is connected to $1+\inkron{D}{\ell}=0,1$ or $2$ \emph{siblings} on the surface, and $\ell-\inkron{D}{\ell}$ \emph{children} on the floor.
An example with $\ell=7$ and $\inkron{D}{\ell}=1$ is shown below:
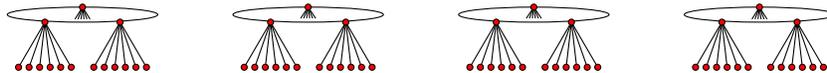
\vspace{-2pt}

\begin{figure}[htp]
\begin{tikzpicture}
\draw (-4.5,0) ellipse (1 and 0.1);
\draw (-1.5,0) ellipse (1 and 0.1);
\draw (1.5,0) ellipse (1 and 0.1);
\draw (4.5,0) ellipse (1 and 0.1);
\draw (-4.5,0.1) -- (-4.6,-0.06);
\draw (-4.5,0.1) -- (-4.56,-0.06);
\draw (-4.5,0.1) -- (-4.52,-0.06);
\draw (-4.5,0.1) -- (-4.48,-0.06);
\draw (-4.5,0.1) -- (-4.44,-0.06);
\draw (-4.5,0.1) -- (-4.4,-0.06);
\draw[fill=red] (-4.5,0.1) circle (0.04);
\draw (-1.5,0.1) -- (-1.6,-0.05);
\draw (-1.5,0.1) -- (-1.56,-0.05);
\draw (-1.5,0.1) -- (-1.52,-0.05);
\draw (-1.5,0.1) -- (-1.48,-0.05);
\draw (-1.5,0.1) -- (-1.44,-0.05);
\draw (-1.5,0.1) -- (-1.4,-0.05);
\draw[fill=red] (-1.5,0.1) circle (0.04);
\draw (1.5,0.1) -- (1.6,-0.05);
\draw (1.5,0.1) -- (1.56,-0.05);
\draw (1.5,0.1) -- (1.52,-0.05);
\draw (1.5,0.1) -- (1.48,-0.05);
\draw (1.5,0.1) -- (1.44,-0.05);
\draw (1.5,0.1) -- (1.4,-0.05);
\draw[fill=red] (1.5,0.1) circle (0.04);
\draw (4.5,0.1) -- (4.60,-0.06);
\draw (4.5,0.1) -- (4.56,-0.06);
\draw (4.5,0.1) -- (4.52,-0.06);
\draw (4.5,0.1) -- (4.48,-0.06);
\draw (4.5,0.1) -- (4.44,-0.06);
\draw (4.5,0.1) -- (4.4,-0.06);
\draw[fill=red] (4.5,0.1) circle (0.04);
\draw (-5,-0.1) -- (-5.35,-0.7);
\draw[fill=red] (-5.35,-0.7) circle (0.04);
\draw (-5,-0.1) -- (-5.21,-0.7);
\draw[fill=red] (-5.21,-0.7) circle (0.04);
\draw (-5,-0.1) -- (-5.07,-0.7);
\draw[fill=red] (-5.07,-.7) circle (0.04);
\draw (-5,-0.1) -- (-4.93,-0.7);
\draw[fill=red] (-4.93,-0.7) circle (0.04);
\draw (-5,-0.1) -- (-4.79,-0.7);
\draw[fill=red] (-4.79,-0.7) circle (0.04);
\draw (-5,-0.1) -- (-4.65,-0.7);
\draw[fill=red] (-4.65,-0.7) circle (0.04);
\draw[fill=red] (-5,-0.1) circle (0.04);

\draw (-4,-0.1) -- (-4.35,-0.7);
\draw[fill=red] (-4.35,-0.7) circle (0.04);
\draw (-4,-0.1) -- (-4.21,-0.7);
\draw[fill=red] (-4.21,-0.7) circle (0.04);
\draw (-4,-0.1) -- (-4.07,-0.7);
\draw[fill=red] (-4.07,-0.7) circle (0.04);
\draw (-4,-0.1) -- (-3.93,-0.7);
\draw[fill=red] (-3.93,-0.7) circle (0.04);
\draw (-4,-0.1) -- (-3.79,-0.7);
\draw[fill=red] (-3.79,-0.7) circle (0.04);
\draw (-4,-0.1) -- (-3.65,-0.7);
\draw[fill=red] (-3.65,-0.7) circle (0.04);
\draw[fill=red] (-4,-0.1) circle (0.04);

\draw (-2,-0.1) -- (-2.35,-0.7);
\draw[fill=red] (-2.35,-0.7) circle (0.04);
\draw (-2,-0.1) -- (-2.21,-0.7);
\draw[fill=red] (-2.21,-0.7) circle (0.04);
\draw (-2,-0.1) -- (-2.07,-0.7);
\draw[fill=red] (-2.07,-0.7) circle (0.04);
\draw (-2,-0.1) -- (-1.93,-0.7);
\draw[fill=red] (-1.93,-0.7) circle (0.04);
\draw (-2,-0.1) -- (-1.79,-0.7);
\draw[fill=red] (-1.79,-0.7) circle (0.04);
\draw (-2,-0.1) -- (-1.65,-0.7);
\draw[fill=red] (-1.65,-0.7) circle (0.04);
\draw[fill=red] (-2,-0.1) circle (0.04);

\draw (-1,-0.1) -- (-1.35,-0.7);
\draw[fill=red] (-1.35,-0.7) circle (0.04);
\draw (-1,-0.1) -- (-1.21,-0.7);
\draw[fill=red] (-1.21,-0.7) circle (0.04);
\draw (-1,-0.1) -- (-1.07,-0.7);
\draw[fill=red] (-1.07,-0.7) circle (0.04);
\draw (-1,-0.1) -- (-0.93,-0.7);
\draw[fill=red] (-0.93,-0.7) circle (0.04);
\draw (-1,-0.1) -- (-0.79,-0.7);
\draw[fill=red] (-0.79,-0.7) circle (0.04);
\draw (-1,-0.1) -- (-0.65,-0.7);
\draw[fill=red] (-0.65,-0.7) circle (0.04);
\draw[fill=red] (-1,-0.1) circle (0.04);

\draw (1,-0.1) -- (1.35,-0.7);
\draw[fill=red] (1.35,-0.7) circle (0.04);
\draw (1,-0.1) -- (1.21,-0.7);
\draw[fill=red] (1.21,-0.7) circle (0.04);
\draw (1,-0.1) -- (1.07,-0.7);
\draw[fill=red] (1.07,-0.7) circle (0.04);
\draw (1,-0.1) -- (0.93,-0.7);
\draw[fill=red] (0.93,-0.7) circle (0.04);
\draw (1,-0.1) -- (0.79,-0.7);
\draw[fill=red] (0.79,-0.7) circle (0.04);
\draw (1,-0.1) -- (0.65,-0.7);
\draw[fill=red] (0.65,-0.7) circle (0.04);
\draw[fill=red] (1,-0.1) circle (0.04);

\draw (2,-0.1) -- (2.35,-0.7);
\draw[fill=red] (2.35,-0.7) circle (0.04);
\draw (2,-0.1) -- (2.21,-0.7);
\draw[fill=red] (2.21,-0.7) circle (0.04);
\draw (2,-0.1) -- (2.07,-0.7);
\draw[fill=red] (2.07,-0.7) circle (0.04);
\draw (2,-0.1) -- (1.93,-0.7);
\draw[fill=red] (1.93,-0.7) circle (0.04);
\draw (2,-0.1) -- (1.79,-0.7);
\draw[fill=red] (1.79,-0.7) circle (0.04);
\draw (2,-0.1) -- (1.65,-0.7);
\draw[fill=red] (1.65,-0.7) circle (0.04);
\draw[fill=red] (2,-0.1) circle (0.04);

\draw (4,-0.1) -- (4.35,-0.7);
\draw[fill=red] (4.35,-0.7) circle (0.04);
\draw (4,-0.1) -- (4.21,-0.7);
\draw[fill=red] (4.21,-0.7) circle (0.04);
\draw (4,-0.1) -- (4.07,-0.7);
\draw[fill=red] (4.07,-0.7) circle (0.04);
\draw (4,-0.1) -- (3.93,-0.7);
\draw[fill=red] (3.93,-0.7) circle (0.04);
\draw (4,-0.1) -- (3.79,-0.7);
\draw[fill=red] (3.79,-0.7) circle (0.04);
\draw (4,-0.1) -- (3.65,-0.7);
\draw[fill=red] (3.65,-0.7) circle (0.04);
\draw[fill=red] (4,-0.1) circle (0.04);

\draw (5,-0.1) -- (5.35,-0.7);
\draw[fill=red] (5.35,-0.7) circle (0.04);
\draw (5,-0.1) -- (5.21,-0.7);
\draw[fill=red] (5.21,-0.7) circle (0.04);
\draw (5,-0.1) -- (5.07,-0.7);
\draw[fill=red] (5.07,-0.7) circle (0.04);
\draw (5,-0.1) -- (4.93,-0.7);
\draw[fill=red] (4.93,-0.7) circle (0.04);
\draw (5,-0.1) -- (4.79,-0.7);
\draw[fill=red] (4.79,-0.7) circle (0.04);
\draw (5,-0.1) -- (4.65,-0.7);
\draw[fill=red] (4.65,-0.7) circle (0.04);
\draw[fill=red] (5,-0.1) circle (0.04);
\end{tikzpicture}
\end{figure}

Provided $h(\O)\ge \ell+2$, this set of $\ell$-volcanoes contains enough information to completely determine $\Phi_\ell\bmod p$.
This is the basis of the algorithm in \cite[Alg.\ 2.1]{BrokerLauterSutherland:CRTModPoly}, which we adapt here.
Selecting a sufficiently large set of such primes $p$ allows one to compute $\Phi_\ell$ over $\Z$ (via the CRT), or modulo an arbitrary prime~$q$ (via the explicit CRT).
In order to achieve the best complexity bounds, it is important to choose both the order $\O$ and the primes $p$ carefully.
We thus introduce the following definitions, in which $c_1$ and $c_2$ are fixed constants that do not depend on $\ell$ or $\O$ (in our implementation we used $c_1=1.5$ and $c_2=256$).

\begin{definition}\label{def:orders}
Let $\O$ be a quadratic order with discriminant $D=u^2D_0<0$, with $D_0$ fundamental, and let $c_1,c_2>1$ be constants.
We say that $\O$ is \textbf{suitable} for $\ell$ if
\begin{align*}
\operatorname{(i)}\  \ell+2\ \le h(\O)\le c_1\ell,\quad \operatorname{(ii)}\ 4 < |D_0|\le c_2^2,\quad \operatorname{(iii)}\ \ell^2 \le |D|\le c_2^2\ell^2,\\
\operatorname{(iv)}\ \gcd(u,2\ell D)=1,\quad \operatorname{(v)}\ q\le \min(c_2,\ell) \text{ for all primes } q|u.
\end{align*}
\end{definition}
This definition combines the criteria in \cite[Def.\ 4.2]{BrokerLauterSutherland:CRTModPoly} and \cite[Thm.\ 5.1]{BrokerLauterSutherland:CRTModPoly}.
Provided that $c_1$ and $c_2$ are not too small, suitable orders exist for every odd prime $\ell$; with $c_1=4$ and $c_2=16$, for example, we may use orders with $D=-7\cdot 3^{2n}$ for all $\ell >3$.  Ideally we want $c_1$ to be as close to 1 as possible, but this makes it harder to find suitable orders.  For the asymptotic analysis, any values of $c_1$ and $c_2$ will do.

\begin{definition}\label{def:primes}
A prime $p$ is \textbf{suitable} for $\ell$ and $\O$ if $p\equiv 1\bmod \ell$ and $p$ satisfies $4p=t^2-\ell^2v^2D$ for some $t,v\in\Z$ with $\ell\ndiv v$ and $\omega(v)\le 2\log(\log v + 3)$.
\end{definition}

The function $\omega(v)$ counts the distinct prime divisors of $v$.
The bound on $\omega(v)$ ensures that if $\O$ is suitable for $\ell$ then many small primes split in $\O$ and do not divide $u$ or $v$.
Such primes allow us to more efficiently enumerate $\cl(\O)$ and $\cl(\O')$.

\subsection{Selecting primes with the GRH}\label{subsection:primes}
In order to apply the isogeny volcano method to compute $\Phi_\ell\bmod q$ (or $\phi_\ell\bmod q$, as we shall do), we need a sufficiently large set $S$ of suitable primes $p$.
We deem $S$ to be sufficiently large whenever $\sum_{p\in S} \log p \ge B + \log 4$, where $B$ is an upper bound on the logarithmic height of the integers whose reductions mod $q$ we wish to compute with the explicit CRT.
For $\Phi_\ell(X,Y)=\sum_{i,j} a_{ij}X^iY^j$, we may bound $h(\Phi_\ell)=\log\max_{i,j}|a_{ij}|$ using
\begin{equation}\label{eq:B1}
h(\Phi_\ell) \le 6\ell\log \ell + 18\ell,\qquad\text{and}\qquad h(\Phi_\ell) \le 6\ell\log\ell + 16\ell + 14\sqrt{\ell}\log\ell,
\end{equation}
as proved in \cite{BrokerSutherland:PhiHeightBound} (we prefer the latter bound when $\ell>3187$).

Heuristically (and in practice), it is easy to construct the set $S$.
Given an order~$\O$ of discriminant~$D$ suitable for $\ell$, we fix $v=2$ if $D\equiv 1\bmod 8$ and $v=1$ otherwise, and for increasing $t\equiv 2\bmod \ell$ of correct parity we test whether $p=(t^2-v^2\ell^2D)/4$ is prime.
We add each prime value of $p$ to $S$, and stop when $S$ is sufficiently large.

Unfortunately, we cannot prove that this method will find \emph{any} primes, even under the GRH.
Instead, we use Algorithm~6.2 in~\cite{BrokerLauterSutherland:CRTModPoly}, which picks an upper bound~$x$ and generates random integers $t$ and $v$ in suitable intervals to obtain candidate primes $p=(t^2-v^2\ell^2D)/4\le x$ that are then tested for primality.
The algorithm periodically increases $x$, so its expected running time is $O(B^{1+\epsilon})$, even without the GRH.
To ensure that the bound on $\omega(v)$ in Definition~\ref{def:primes} is satisfied, unsuitable $v$'s are discarded; this occurs with negligible probability.

Under the GRH, there are effective constants $c_3,c_4 > 0$ such that $x\ge c_3\ell^6\logpow{4}{\ell}$ guarantees at least $c_4\ell^3\logpow{3}{\ell}$ suitable primes less than $x$, by \cite[Thm.\ 4.4]{BrokerLauterSutherland:CRTModPoly}.
Asymptotically, this is far more than the $O(\ell)$ primes we need to compute $\Phi_\ell\bmod q$.
Here we may consider larger values of $B$, and in general, $x=O(B^2 + \ell^6\logpow{4}{\ell})$ suffices.
We note that $S$ contains $O(B/\log B)$ primes (unconditionally), and under the GRH we have $\log p = O(\log B+\log\ell)$ for all $p\in S$.

\section{Algorithms}\label{section:algorithms}
Let $q$ be a prime and let $E$ be an elliptic curve over $\Fq$.
A simple algorithm to compute $\phi_\ell(Y)=\Phi_\ell(j(E),Y)\in\Fq[Y]$ with the explicit CRT works as follows.
Let~$\jh$ be the integer in $[0,q-1]$ corresponding to $j(E)\in\Fq\simeq\Z/q\Z$.
For a sufficiently large set $S$ of suitable primes $p$, compute $\Phi_\ell(X,Y)\bmod p$ using the isogeny volcano algorithm and evaluate $\Phi_\ell(\jh,Y)\bmod p$ to obtain $\bar\phi_\ell\in\Fp[Y]$, and use the explicit CRT mod $q$ to eventually obtain $\phi_\ell\in\Fq[Y]$.

This na\"ive algorithm suffers from two significant defects.
The most serious is that we may now require a much larger set $S$ than is needed to compute $\Phi_\ell\bmod q$.
Compared to the coefficients of $\Phi_\ell$, which have height $h(\Phi_\ell)=O(\ell\log\ell)$ bounded by \eqref{eq:B1},
we now need to use the $O(\ell\log\ell+\ell\log q)$ bound
\begin{equation}\label{eq:B2}
h(\Phi_\ell(\jh,Y)) \le  h(\Phi_\ell) + (\ell+1)\log q + \log (\ell+2),
\end{equation}
since $\Phi_\ell(\jh,Y)$ involves powers of $\jh$ up to $\jh^{\ell+1}$.

If $\log q$ is comparable to $\log \ell$, then the difference between the bounds in \eqref{eq:B1} and~\eqref{eq:B2} may be negligible.
But when $\log q$ is comparable to $\ell$, using the bound in~\eqref{eq:B2} increases the running time dramatically.
This issue is addressed by Algorithm~1.

The second defect of the na\"ive algorithm is that although its space complexity may be significantly better than the $O(\ell^2\log q)$ space required to compute $\Phi_\ell\bmod q$, it is still quasi-quadratic in $\ell$.
But the size of $\phi_\ell$ is linear in $\ell$, so we might hope to do better,
and indeed we can.  This is achieved by Algorithm~2.

A hybrid approach that combines aspects of both algorithms is discussed in \S\ref{subsection:hybrid}.

\subsection{Algorithm 1}\label{subsection:alg1}
The increase in the height bound from \eqref{eq:B1} to \eqref{eq:B2} is caused by the fact that \emph{we are exponentiating in the wrong ring}.
Rather than lifting $j(E)\in\Fq$ to the integer $\jh$ and computing powers of its reduction in $\Fp$ (which simulates powering in $\Z$), we should instead compute powers $j(E), j(E)^2, \ldots, j(E)^{\ell+1}$ in $\Fq$, lift these values to integers $\hat x_1,\hat x_2,\ldots,\hat x_{\ell+1}$, and work with their reductions in $\Fp$, as in \cite[\S4.4]{Sutherland:CMmethod} (a similar strategy is used in \cite{KedlayaUmans:ModularComposition}).
Of course the reductions of $\hat x_1, \hat x_2,\ldots, \hat x_{\ell+1}$ need not correspond to powers of any particular element in $\Fp$;  nevertheless, if we simply replace each occurrence of~$X^i$ in the modular polynomial $\Phi_\ell(X,Y)\bmod p$ with $\hat x_i\bmod p$, we achieve the same end result using a much smaller height bound.

We now present Algorithm~1 to compute $\phi(Y)=\phi_\ell(Y)=\Phi_\ell(j(E),y)$.  If desired, the algorithm can also compute the polynomials
$\phi_X(Y)=(\partial\Phi_\ell/\partial X)(j(E),Y)$ and $\phi_{XX}(Y)=(\partial^2\Phi_\ell/\partial X^2)(j(E),Y)$, which may be used to compute normalized isogenies, as described in \S\ref{subsection:normalized}.  These optional steps are shown in parentheses.
\smallskip

\noindent
\textbf{Algorithm 1}\\
\textbf{Input:} An odd prime $\ell$, a prime $q$, and $j(E)\in\Fq$.\\
\textbf{Output:} The polynomial $\phi(Y)=\Phi_\ell(j(E),Y)\in\Fq[Y]$ (and $\phi_X(Y)$ and $\phi_{XX}(Y)$).
\begin{enumerate}[1.]
\item Select an order $\O$ suitable for $\ell$ and a set of suitable primes $S$ (see~\S\ref{subsection:primes}),\\
using the height bound $B=6\ell\log\ell+18\ell+\log q + 3\log(\ell+2)$.
\item Compute the Hilbert class polynomial $H_\O(X)$ via \cite[Alg.\  2]{Sutherland:HilbertClassPolynomials}.
\item Perform CRT precomputation mod $q$ using $S$ (see~\S\ref{subsection:crt}).
\item Compute integers $\hat x_i\in[0,q-1]$ such that $\hat x_i\equiv j(E)^i\bmod q$, for $0\le i\le\ell+1$.
\item For each prime $p\in S$:
\begin{enumerate}[a.]
\item Compute $\Phi_\ell(X,Y)\bmod p$ using $H_\O$, via \cite[Alg.\ 2.1]{BrokerLauterSutherland:CRTModPoly}.
\item Compute $\bar\phi(Y)=\sum_{i,j} a_{ij}\hat x_iY^j \bmod p$, where $\Phi_\ell(X,Y) = \sum_{i,j}a_{ij}X^iY^j$.
\item (Compute $\bar\phi_X(Y)=\sum_{i,j} ia_{ij}\hat x_iY^j \bmod p$\\
\phantom{}\hspace{-2pt} and also $\bar\phi_{XX}(Y)=\sum_{i,j} i(i-1)a_{ij}\hat x_iY^j \bmod p$).
\item Update CRT sums for each coefficient $c_i$ of $\bar\phi$ (and of $\bar\phi_X$ and $\bar\phi_{XX}$).
\end{enumerate}
\item Perform CRT postcomputation to obtain $\phi$ (and $\phi_X$ and $\phi_{XX}$) mod $q$.
\item Output $\phi$ (and $\phi_X$ and $\phi_{XX}$).
\end{enumerate}

\begin{proposition}
The output $\phi(Y)$ of Algorithm~1 is equal to $\Phi_\ell(j(E),Y)$.\\
$($and $\phi_X(Y)=(\partial\Phi_\ell/\partial X)(j(E),Y)$ and $\phi_{XX}(Y)=(\partial^2\Phi_\ell/\partial X^2)(j(E),Y))$.
\end{proposition}
\begin{proof}
Let $\varphi=\Phi_\ell(\jh,Y)\in\Fq[Y]$.
Let $\hat x_i\in \Z$ be as in step 4.
Write $\Phi_\ell$ as $\sum_{i,j}a_{ij}X^iY^j$, with $a_{ij}\in\Z$ and let $\hat\phi =\sum_{i,j} a_{ij}\hat x_iY^j\in\Z[Y]$.
Then $\varphi\equiv\hat\phi\bmod q$, and $\bar\phi\equiv\hat\phi\bmod p$.
To prove $\phi=\varphi$, we only need to show $h(\hat\phi)\le B$.
We have
\[
\bigl|\textstyle{\sum_i} a_{ij}\hat x_i\bigr| \le (\ell+2)q\exp h(\Phi_\ell),
\]
for $0\le j \le \ell+1$, hence $h(\hat\phi)\le B$.
The proofs for $\phi_X$ and $\phi_{XX}$ are analogous.
We note that the last term in $B$ can be reduced to $\log(\ell+2)$ if $\phi_X$ and $\phi_{XX}$ are not being computed.
\end{proof}

\begin{theorem}\label{thm:alg1comp}
Assume the GRH.  The expected running time of Algorithm~1 is $O(\ell^2B\log^2B\llog B)$, where $B=O(\ell\log\ell +\log q)$ is as specified in step 1.
It uses $O(\ell\log q +\ell^2\log B)$ space.
\end{theorem}
\begin{proof}
We use $\M(n)=O(n\log n\llog n)$ to denote the cost of multiplication \cite{Schonhage:Multiplication}.
For step 1, we assume the time spent selecting $\O$ is negligible (as noted in \S\ref{subsection:modpoly}, one may simply choose orders with discriminants of the form $D=-7\cdot 3^{2n}$),
and under the GRH the expected time to construct $S$ is $O(B^{1+\epsilon})$, using $O(B)$ space, as explained in \S\ref{subsection:primes}.
Step 2 uses $O(\ell^{2+\epsilon})$ expected time and $O(\ell(\log \ell + \log q))$ space, by \cite[Thm.\ 1]{Sutherland:HilbertClassPolynomials}, since $h(D)=O(\ell)$.
An analysis as in \cite[\S 6.3]{Sutherland:HilbertClassPolynomials} shows that the total cost of all CRT operations is $O(\ell \M(B)\log B)$ time and $O(\ell \log q)$ space.
Step~4 uses $O(\ell\M(\log q))$ time and $O(\ell\log q)$ space.

The set $S$ contains $O(B/\log B)$ primes $p$, and under the GRH, $\log p = O(\log B)$; see \S\ref{subsection:primes}.
The cost per $p$ is dominated by step 5a, which takes $O(\ell^2\log^3 B\llog B)$ expected time and $O(\ell^2\log B)$ space, by \cite{BrokerLauterSutherland:CRTModPoly}.
This yields an $O(\ell^2B\log^2B\llog B)$ bound for step 5, which dominates, and the total space is $O(\ell\log q + \ell^2\log B)$.
\end{proof}

When $\log q=\Theta(\ell)$, the time bound in Theorem~\ref{thm:alg1comp} reduces to $O(\ell^3\logpow{3}{\ell}\llog \ell)$, the same as the time to compute $\Phi_\ell\bmod q$, and the space bound is $O(\ell\log\ell\log q)$, which is within an $O(\log\ell)$ factor of the best possible.

\subsection{Algorithm 2}\label{subsection:alg2}

We now present Algorithm~2, which for $q > \ell$ has optimal space complexity $O(\ell\log q)$.
When $q$ is reasonably small, say $\log q = o(\logpow{2}{\ell})$, Algorithm~2 is also faster than Algorithm~1, but when $\log q$ is large it may be much slower, since it uses the same height bound \eqref{eq:B2} as the na\"ive approach (see \S\ref{subsection:hybrid} for a hybrid approach).
The computation of $\bar\phi\in\Fp[Y]$ is more intricate, so we present it separately as Algorithm~2.1.
Unlike Algorithm~1, it is not so easy to also compute $\phi_X$ and $\phi_{XX}$, but an alternative method to compute normalized isogenies using Algorithm~2 is given in \S\ref{subsection:normalized}.
\bigskip

\noindent
\textbf{Algorithm 2}\\
\textbf{Input:} An odd prime $\ell$, a prime $q$, and $j(E)\in\Fq$.\\
\textbf{Output:} The polynomial $\phi(Y)=\Phi_\ell(j(E),Y)\in\Fq[Y]$.\\
\vspace{-8pt}
\begin{enumerate}[1.]
\item Select an order $\O$ suitable for $\ell$ and a suitable set of primes $S$ (see~\S\ref{subsection:primes}),\\
using the height bound $B=6\ell\log\ell+18\ell+(\ell+1)\log q +\log(\ell+2)$.
\item Compute the Hilbert class polynomial $H_\O$ via \cite[Alg.\  2]{Sutherland:HilbertClassPolynomials}.
\item Perform precomputation for the explicit CRT mod $q$ using $S$.
\item Let $\jh$ be the integer in $[0,q-1]$ congruent to $j(E)\bmod q$.
\item For each prime $p\in S$:
\begin{enumerate}[a.]
\item Compute $\bar\phi(Y)=\Phi_\ell(\jh,Y)\bmod p$ using $\O$ and $H_\O$ via Algorithm~2.1.
\item Update CRT sums for each coefficient $c_i$ of $\bar\phi$.
\end{enumerate}
\item Perform postcomputation for the explicit CRT to obtain $\phi\in\Fq[X]$.
\item Output $\phi$.
\end{enumerate}

\begin{proposition}
The output $\phi(Y)$ of Algorithm~2 is equal to $\Phi_\ell(j(E),Y)$.
\end{proposition}
\begin{proof}
This follows immediately from Proposition~\ref{prop:alg2.1correct} below and the bound
\[
h(\Phi_\ell(\jh,Y))=\log \textstyle{\max_j \bigl|\sum_i a_{ij}\jh^i\bigr|}\le \log(\ell+2)+(\ell+1)\log q+h(\Phi_\ell)\le B.
\]
on the height of $\Phi_\ell(\jh,Y)\in\Z[Y]$.
\end{proof}

\begin{theorem}\label{thm:alg2comp}
Assume the GRH and that $\log q=O(\ell^k)$ for some constant $k$.
The expected running time of Algorithm~2 is $O(\ell^3(\log q + \log \ell)\log\ell\llogpow{2}{\ell}\lllogpow{2}{\ell})$ and it uses $O(\ell\log q+\ell\log \ell)$ space.
\end{theorem}
\begin{proof}
As in the proof of Theorem~\ref{thm:alg1comp}, the expected running time is dominated by the time to compute $\bar\phi(Y)$, which by Theorem~\ref{thm:alg21comp} is $O(\ell^2\log^2 p\llog^2 p\lllog^2 p)$.
There are $O(B/\log B)$ primes $p\in S$, and under the GRH we have $\log p = O(\log B)=O(\log \ell)$.
The space complexity is dominated by the $O(B)=O(\ell\log\ell+\ell\log q)$ size of $S$.
\end{proof}

\subsection{Algorithm~2.1}\label{subsection:alg2.1}

The algorithm in \cite[Alg.~2.1]{BrokerLauterSutherland:CRTModPoly} computes $\Phi_\ell\bmod p$ by enumerating the sets $\Ell_\O(\Fp)$ and $\Ell_{\O'}(\Fp)$, where $\O'= \Z+\ell\O$, the latter of which contains approximately $\ell^2$ elements.  To achieve a space complexity that is quasi-linear in $\ell$, we cannot afford to store the entire set $\Ell_{\O'}(\Fp)$.
We must compute $\Phi_\ell(\jh,Y)\bmod p$ using an \emph{online algorithm}, processing each $j_k\in\Ell_{\O'}(\Fp)$ as we enumerate it, and then discarding it.  Let us consider how this may be done.

Let $y_1,\ldots,y_{h(\O)}$ be the elements of $\Ell_\O(\Fp)$, as enumerated using a polycyclic presentation $\vec{\alpha}$ for $\cl(\O)$.
Each $y_i$ is $\ell$-isogenous to a set $S_i$ of \emph{siblings} in $\Ell_\O(\Fp)$, and to a set $C_i$ of \emph{children} in $\Ell_{\O'}(\Fp)$; see \S\ref{subsection:modpoly}.
Thus we have\vspace{-2pt}
\[
\Phi_\ell(X,y_i)=\Bigl(\prod_{\jt\in S_i}(X-\jt)\Bigr)\Bigl(\prod_{\jt\in C_i}(X-\jt)\Bigr).\vspace{-2pt}
\]
The siblings can be readily identified in our enumeration of $\Ell_\O(\Fp)$ using the CM action (see \S\ref{subsection:cm}).
To identify the children, we need to be able to determine, for any given $j\in\O'$, the set $C_i$ in which it lies.
Each $C_i$ is a subset of the torsor $\Ell_{\O'}(\Fp)$ corresponding to a coset of the subgroup $C\subset\cl(\O')$ generated by the ideals of norm $\ell^2$; indeed, two children have the same parent if and only if they are related by an isogeny of degree $\ell^2$
(the composition of two $\ell$-isogenies).

The group $\cl(\O')$ acts on the cosets of $C$, and we need to compute this action explicitly in terms of the polycyclic presentation $\vec{\beta}$ used to enumerate $\cl(\O')$.
This problem is addressed by a generic group algorithm in the next section that
computes a polycyclic presentation $\vec{\gamma}$ for the quotient $\cl(\O')/C$, along with the $\vec{\gamma}$-representation of the image of each generator in $\vec{\beta}$.

As we enumerate the elements $j_k$ of $\Ell_{\O'}(\Fp)$, starting from a child $j_1$ of $y_1$ obtained via V\'elu's algorithm, we keep track of the element $\delta_k\in\cl(\O')$ whose action sends $j_1$ to $j_k$.
The image of $\delta_k$ in $\cl(\O')/C$ is the coset of $C$ corresponding to the set $C_i$ containing $j_k$, and we simply identify $C_i$ with the $i$th element of $\cl(\O')/C$ as enumerated by $\vec{\gamma}$ (in the lexicographic ordering of $\vec\gamma$-representations).

Thus we can compute the polynomials $\phi_i(X) = \Phi_\ell(X,y_i)$ as we enumerate $\Ell_{\O'}(\Fp)$ by accumulating a partial product of linear factors for each $\phi_i$.
But since our goal is to evaluate $z_i=\phi_i(\jh)\bmod p$, we simply substitute $x=\jh\bmod p$ into each linear factor, as we compute it, and accumulate the partial product in $z_i$.

Having computed the values $z_i$ for $1\le i\le \ell+2$, we interpolate the unique polynomial $\phi(Y)$ of degree at most $\ell+1$ for which $\phi(y_i)=z_i$, using Lagrange interpolation.  This polynomial must be $\Phi_\ell(\jh,Y)$.
We now give the algorithm.
\smallskip

\noindent
\textbf{Algorithm 2.1}\\
\textbf{Input:} An odd prime $\ell$, a suitable order $\O$, a suitable prime $p$, and $x\in\Fp$.\\
\textbf{Output:} The polynomial $\phi(Y)=\Phi_\ell(x,Y)\in\Fp[Y]$.\\
\vspace{-8pt}
\begin{enumerate}[1.]
\item Compute presentations $\vec{\alpha}$ of $\cl(\O)$ and $\vec{\beta}$ of $\cl(\O')$ suitable for $p$.
\item Represent generators of the subgroup $C\subset\cl(\O')$ defined above in terms of $\vec\beta$.
\item Compute the presentation $\vec\gamma$ of $\cl(\O')/C$ derived from $\vec{\beta}$, via Algorithm~3.
\item Find a root $w_1$ of $H_\O\bmod p$ (compute $H_\O \bmod p$ if needed).
\item Enumerate $\Ell_\O(\Fp)$ as $w_1,w_2,\ldots,w_{h(\O)}$ using $\vec{\alpha}$. 
\item Obtain $j_1\in\Ell_{\O'}(\Fp)$ from $w_1$ using V\'elu's algorithm.
\item Set $z_i\leftarrow 1$ and $y_i\leftarrow \texttt{null}$ for $1\le i\le \ell+2$.
\item For each $j_k = \delta_kj_1$ in $\Ell_{\O'}(\Fp)$ enumerated using $\vec{\beta}$:
\begin{enumerate}[a.]
\item Compute the index $i$ of $\delta_k$ in the $\vec{\gamma}$-enumeration of $\cl(\O')/C$.\\
If $i>\ell+2$ then proceed to the next $j_k$, skipping steps b and c below.
\item If $y_i=\texttt{null}$ then set $y_i$ to the $\ell$-parent of $j_k$ (via V\'elu's algorithm) and for each $\ell$-sibling $\jt$ of $y_i$ in $\Ell_\O(\Fp)$ set $z_i\leftarrow z_i(x-\jt)$.
\item Set $z_i\leftarrow z_i(x-j_k)$.
\end{enumerate}
\item Interpolate $\phi\in \Fp[Y]$ such that $\deg\phi\le\ell+1$ and $\phi(y_i)=z_i$ for $1\le i\le \ell+2$.
\item Output $\phi$.
\end{enumerate}

The value \texttt{null} assigned to $y_i$ in step 7 is used to indicate that the value of $y_i$ is not yet known.
Each $y_i$ is eventually set to a distinct $w_j\in\Ell_\O(\Fp)$.

\begin{remark}
In practical implementations, Algorithm~2 selects the primes $p\in S$ so that the presentations $\vec\alpha$, $\vec\beta$, and $\vec\gamma$ are the same for every $p$ and and precomputes them (the only reason they might not be the same is the presence of prime ideals whose norm divides $v=v(p)$, but in practice we fix $v\le 2$, as discussed in \S\ref{subsection:primes}).
\end{remark}

\begin{proposition}\label{prop:alg2.1correct}
Algorithm~2.1 outputs $\phi(Y)=\Phi_\ell(x,Y)\bmod p$.
\end{proposition}
\begin{proof}
Let  $\varphi(Y)=\Phi_\ell(x,Y)$.
It follows from the discussion above that Algorithm~2.1 computes $z_i=\Phi_\ell(x,y_i)$ for $1\le i\le \ell+2$.
Thus $\phi(y_i)=z_i=\varphi(y_i)$ for $\ell+2$ values $y_i\in \Ell_\O(\Fp)$, and these values are necessarily distinct.
The polynomials $\phi$ and $\varphi$ both have degree at most $\ell+1$, therefore $\phi=\varphi$.
\end{proof}

\begin{theorem}\label{thm:alg21comp}
Assume the GRH.
Algorithm~2.1 runs in $O(\ell^2n^2\logpow{2}{n}\llogpow{2}{n})$ expected time using $O(\ell n)$ space, where $n=\log p$.
\end{theorem}
\begin{proof}
The time complexity is dominated by step 8, which enumerates the $O(\ell^2)$ elements of $\Ell_{\O'}(\Fp)$ using $\vec\beta$.
By \cite[Thm.\ 5.1]{BrokerLauterSutherland:CRTModPoly} and the suitability of $\O$ and $p$, we may assume each $\beta_i=[\mathfrak b_i]$, where $\mathfrak b_i$ has prime norm $b_i = O(\log n\llog n)$.
Using Kronecker substitution and probabilistic root-finding \cite{Gathen:ComputerAlgebra}, the expected time to find the (at most $2$) roots of $\Phi_{b_i}(j_k,Y)$ is $O(n\M(n\log n\llog n))$, which dominates the cost for each $j_k$.
Applying $\M(n)=O(n\log n\llog n)$ and multiplying by $\ell^2$ yields the desired time bound.
Taking into account $h(\O)=O(\ell)$ and $p>\ell$, the computation of $H_\O\bmod p$ uses $O(\ell n)$ space, by \cite[Thm.~1]{Sutherland:HilbertClassPolynomials}, and this bounds the total space.
\end{proof}

\subsection{A hybrid approach}\label{subsection:hybrid}
Algorithm~2 achieves an essentially optimal space complexity, but its time complexity is attractive only when $\log q$ is not too large, say $\log q = O(\log^2 \ell)$.
Algorithm~1 has an excellent time complexity, but achieves an optimal space complexity only when $\log q$ is very large, say $\log q = \Omega(\ell \log\ell)$.
To address the intermediate range, we present a hybrid approach suggested by Daniel Kane that has the same space complexity as Algorithm~2 and a time complexity that is within a polylogarithmic factor of the time complexity of Algorithm~1.

The strategy is to replace the computation of $\bar\phi(Y)=\sum_{i,j} a_{ij}\hat{x}_iY^j\bmod p$ in step~5 of Algorithm~1 with Algorithm 2.2 below.  Algorithm 2.2 is similar to Algorithm 2.1, but rather than accumulating $\ell+2$ values $z_i$ in parallel, we  compute them individually by enumerating the each of the sets $C_i$ of children $y_i$ in turn.

\smallskip

\noindent
\textbf{Algorithm 2.2}\\
\textbf{Input:} An odd prime $\ell$, suitable order $\O$, suitable prime $p$, and $x_1,\ldots,x_{\ell+1}\in\Fp$.\\
\textbf{Output:} $\phi(Y)=\sum_{i,j}a_{ij}x_iY^j\in\Fp[Y]$, where $\Phi_\ell(X,Y)=\sum_{i,j} a_{ij}X^iY^j$.\\
\vspace{-8pt}
\begin{enumerate}[1.]
\item Compute presentations $\vec{\alpha}$, $\vec{\beta}$, and $\vec{\gamma}$ as in Algorithm 2.1.
\item Find a root $y_1$ of $H_\O\bmod p$ (compute $H_\O \bmod p$ if needed).
\item Enumerate $\Ell_\O(\Fq)$ as $y_1,y_2,\ldots,y_{h(\O)}$ using $\vec{\alpha}$.
\item Obtain $j_1\in\Ell_{\O'}(\Fq)$ from $y_1$ using V\'elu's algorithm.
\item For $i$ from $1$ to $\ell+2$ do the following:
\begin{enumerate}[a.]
\item Use $\vec{\alpha}$ to compute the set $S_i$ of siblings of $y_i$ in $\Ell_\O(\Fp)$.
\item Use $\vec{\beta}$ and $\vec{\gamma}$ to compute the set $C_i$ of children of $y_i$ in $\Ell_{\O'}(\Fp)$ (see below).
\item Compute $\phi_i(X)=\prod_{\jt\in S_i}(X-\jt)\prod_{\jt\in C_i}(X-\jt) = \sum c_{ik} X^k \bmod p$.
\item Compute $z_i = \sum_k c_{ik} x_k$.
\end{enumerate}
\item Interpolate $\phi\in \Fp[Y]$ such that $\deg\phi\le\ell+1$ and $\phi(y_i)=z_i$ for $1\le i\le \ell+2$.
\item Output $\phi$.
\end{enumerate}

To compute the set $C_i$ in step 5b, for each $\jt\in C_i$ we determine the $\delta\in\cl(\O')$ for which $\jt=\delta j_1$.
Under the GRH, it follows from \cite[Thm.\ 2.1]{Childs:QuantumIsogenies} that we can express $\delta$ in the form $\delta=[\mathfrak p_1\cdots\mathfrak p_t]$, where the ideals $\mathfrak p_i$ have prime norms bounded by $\log^c\ell$, for any $c>2$, with $t = O(\log \ell)$.  Assuming $\log p = O(\log\ell)$, this implies that we can compute each $\jt$ in $O(\log^{6+\epsilon})$ expected time, for any $\epsilon > 0$.

\begin{proposition}
Algorithm 2.2 outputs $\phi(Y)=\sum_{i,j}a_{ij}x_iY^j$.
\end{proposition}
\begin{proof}
Let $\varphi(y) = \sum_{i,j}a_{ij}x_iY^j$.
The roots of $\phi_i(X)$ are the roots of $\Phi_\ell(X,y_i)$, thus $\sum_k{c_{ik}}X_k=\sum_{k,j}a_{kj}X^ky_i^j$, and we have
$c_{ik}=\sum_ja_{kj}y_i^j$.  It follows that $\phi(y_i)=z_i=\sum_{k}\sum_ja_{kj}x_ky_i^j=\varphi(y_i)$.
Since $\phi(Y)$ and $\varphi(Y)$ both have degree at most $\ell+1$ and agree at $\ell+2$ distinct values $y_i$, they must be equal.
\end{proof}

\begin{theorem}
Assume the GRH and that $\log q=O(\ell\log\ell)$.
If Algorithm 1 uses Algorithm 2.2 to compute $\bar\phi(Y)$ in step 5, its expected running time is $O(\ell^3\log^{6+\epsilon}\ell)$ using $O(\ell\log q +\ell\log\ell)$ space.
\end{theorem}
\begin{proof}
It suffices to show that if $\log p = O(\log\ell)$, then Algorithm 2.2 runs in $O(\ell^2\log^{6+\epsilon}\ell)$ expected time using $O(\ell\log\ell)$ space.  The space bound is clear.  For the time bound, the cost of step 5b is $O(\ell\log^{6+\epsilon}\ell)$ (see above), and this yields an $O(\ell^2\log^{6+\epsilon}\ell)$ bound on the expected time for step 5, which dominates.
\end{proof}

The extra logarithmic factors make the hybrid approach significantly slower than Algorithm~1 in practice, but it does allow us to achieve an essentially optimal space complexity with a quasi-cubic running time across the entire range of parameters.

\subsection{Computing a polycyclic presentation for a quotient group}\label{subsection:alg3}
We now give a generic algorithm to derive a polycyclic presentation $\vec\gamma$ for a quotient of finite abelian groups $G/H$.
This presentation can be used to efficiently compute in $G/H$, and to compute the image of elements of $G$, as required by Algorithm~2.1.
\smallskip

\noindent
\textbf{Algorithm 3}\\
\textbf{Input:} A minimal polycyclic presentation $\vec\beta=(\beta_1,\ldots,\beta_k)$ for a finite abelian group $G$ and a subgroup $H=\langle \alpha_1,\ldots,\alpha_t\rangle$, with each $\alpha_i$ specified in terms of $\vec\beta$.\\
\textbf{Output:} A polycyclic presentation $\vec\gamma$ for $G/H$, with $\gamma_i=[\beta_i]$ for each $\beta_i\in\vec\beta$.\\
\vspace{-8pt}
\begin{enumerate}[1.]
\item Derive a polycyclic presentation $\vec{\alpha}$ for $H$ from $\alpha_1,\ldots,\alpha_t$ via \cite[Alg.\ 2.2]{Sutherland:HilbertClassPolynomials}.
\item Enumerate $H$ using $\vec{\alpha}$ and create a lookup table $T_H$ to test membership in $H$.
\item Derive a polycyclic presentation $\vec\gamma$ for $G/H$ from $[\beta_1],\ldots,[\beta_k]$ via \cite[Alg.\ 2.2]{Sutherland:HilbertClassPolynomials}, using $T_H$ as described below.
\item Output $\vec\gamma$, with relative orders $r(\vec\gamma)$ and relations $s(\vec\gamma)$.
\end{enumerate}
\smallskip

The polycyclic presentation $\vec\gamma$ output by Algorithm~3 is not necessarily minimal.
It can be converted to a minimal presentation by simply removing those $\gamma_i$ with $r(\gamma_i)=1$, however, for the purpose of computing the image in $G/H$ of elements of~$G$ represented in terms of $\vec\beta$, it is better not to do so.

The algorithm in \cite[Alg.\ 2.2]{Sutherland:HilbertClassPolynomials} requires a \textsc{TableLookup} function that searches for a given group element in a table of distinct group elements.
In Algorithm~3 above, the elements of $G$ are uniquely represented by their $\vec\beta$-representations, but elements of $G/H$ are represented as equivalence classes $[\delta]$, with $\delta\in G$, which is not a unique representation.
To implement the \textsc{TableLookup} function for $G/H$, we do the following:
given $[\delta_0]\in G/H$ and a table $T_{G/H}$ of distinct elements $[\delta_i]$ in $G/H$, we test whether $\delta_0\delta_i^{-1}\in H$, for each $[\delta_i]\in T$.
With a suitable implementation of~$T_H$ (such as a hash table or balanced tree), membership in $H$ can be tested in $O(\log|G|)$ time, which is dominated by the $O(\logpow{2}{|G|})$ time to compute $\delta_0\delta_i^{-1}$.

The problem of uniquely representing elements of $G/H$ is solved once Algorithm~3 completes: every element of $G/H$ has a unique $\vec\gamma$-representation.

\begin{theorem}
Algorithm~3 runs in $O(n\logpow{2}{n})$ time using $O((m+n/m)\log n)$ space, where $n=|G|$ and $m=|H|$.
\end{theorem}
\begin{proof}
The time complexity is dominated by the $n/m$ calls to the \textsc{TableLookup} function performed by
\cite[Alg.\ 2.2]{Sutherland:HilbertClassPolynomials} in step 3, each of which performs $m$ operations in $G$ (using $\vec\beta$-representations) and $m$ lookups in $T_H$, yielding a total cots of $O(n\logpow{2}{n})$.
The space bound is the size of $T_H$ plus the size of $T_{G/H}$.
\end{proof}

\subsection{Other modular functions}\label{subsection:modular functions}
For a modular function $g$ of level $N$ and a prime $\ell\ndiv N$, the modular polynomial $\Phi_\ell^g$ is the minimal polynomial of the function $g(\ell z)$ over the field $\C(g)$.
For suitable functions $g$, the isogeny volcano algorithm for computing $\Phi_\ell(X,Y)$ can be adapted to compute $\Phi_\ell^g(X,Y)$,
as described in \cite[\S 7]{BrokerLauterSutherland:CRTModPoly}.
There are some restrictions: $\Phi_\ell^g$ must have degree $\ell+1$ in both~$X$ and~$Y$, and we require some additional constraints on the suitable orders $\O$ that we use.
Specifically, we require that there is a generator $\tau$ of $\O$ for which $g(\tau)$ lies in the ring class field $K_\O$.
In this case we say that $g(\tau)$ is a \emph{class invariant}, and let $H^g_\O(X)$ denote its minimal polynomial over $K$; see \cite{Broker:pAdicClassInvariants,Enge:FloatingPoint,EngeSutherland:CRTClassInvariants} for algorithms to compute $H^g_\O(X)$.
We also require the polynomial $H^g_\O$ to be defined over $\Z$.

With this setup, there is then a one-to-one correspondence between the roots $j(\tau)$ of $H_\O$ and the roots $g(\tau)$ of $H^g_\O$ in which $\Psi^g(g(\tau),j(\tau))=0$, where $\Psi^g$ is the minimal polynomial of $g$ over $\C(j)$; note that $\Psi^g$ does not depend on $\ell$ and is assumed to be given.
The class group $\cl(\O)\simeq\Gal(K_\O/K)$ acts compatibly on both sets of roots, and this allows us to compute $\Phi_\ell^g$ modulo suitable primes $p$ using essentially the same algorithm that is used to compute $\Phi_\ell\bmod p$.
In particular, we can enumerate the set $\Ell^g_{\O}(\Fp)=\{x\in \Fp: H^g_\O(x)=0\}$ using a polycyclic presentation $\vec\alpha$ for $\cl(\O)$, provided that we exclude from $\vec\alpha$ generators whose norm divides the level of $g$, and similarly for $\Ell^g_{\O'}(\Fp)$, where $\O'=\Z+\ell\O$.

Thus Algorithms~1 and~2 can both be adapted to compute instantiated modular polynomials $\phi^g(Y)=\Phi^g_\ell(x,Y)\bmod q$. 
Some effort may be required to determine the correspondence between $\Ell_\O(\Fp)$ and $\Ell_\O^g(\Fp)$ in cases where $\Psi^g(X,j(E))$ or $\Psi^g(g(E),Y)$ has multiple roots in $\Fp$; this issue arises when we need to compute a child or parent using V\'elu's algorithm. 
There are several techniques for resolving such ambiguities, see \cite[\S 7.3]{BrokerLauterSutherland:CRTModPoly} and especially \cite{EngeSutherland:CRTClassInvariants}, which explores this issue in detail.

We emphasize that the point $x$ at which we are evaluating $\Phi_\ell^g(x,Y)$ may be \emph{any} element of $\Fq$, it need not correspond to the ``$g$-invariant" of an elliptic curve.\footnote{Every $x\in\Fq$ is $j(E)$ for some $E/\Fq$, and when $E$ is ordinary, $j(E)$ is the reduction of some $j(\tau)=j(\hat E)$ with $\Z[\tau]=\O\simeq\End(E)$.  But $g(\tau)$ might not be a class invariant for this $\O$.}
This permits a very useful optimization that speeds up our original version of Algorithm~1 for computing $\phi_\ell(Y)=\phi_\ell^j(Y)$ by a factor of at least~9, as we now explain.

\subsection{Accelerating the computation of $\phi_\ell(Y)$ using $\gamma_2$}\label{subsection:fasterj}
Let $\gamma_2(z)$ be the unique cube-root of $j(z)$ with integral Fourier expansion,
a modular function of level 3 that yields class invariants for $\O$ whenever $3\ndiv\disc(\O)$.
As noted in \cite[\S 7.2]{BrokerLauterSutherland:CRTModPoly}, for $\ell > 3$ the modular polynomial $\Phi_\ell^{\gamma_2}$ can be written as
\begin{equation}\label{eq:phig2}
\Phi_\ell^{\gamma_2}(X,Y) = R(X^3,Y^3)Y^e+S(X^3,Y^3)XY + T(X^3,Y^3)X^2Y^{2-e},
\end{equation}
with $e=\ell+1\bmod 3$ and $R,S,T\in\Z[X,Y]$.  We then have the identity
\begin{equation}\label{eq:RST}
\Phi_\ell = R^3Y^e+(S^3-3RST)XY+TX^2Y^{2-e}.
\end{equation}
When computing $\Phi^{\gamma_2}_\ell\bmod p$ with the isogeny volcano algorithm, one can exploit ~\eqref{eq:phig2} to speed up the computation by at least a factor of 3.
In addition, the integer coefficients of $\Phi_\ell^{\gamma_2}$ are also smaller than those of $\Phi_\ell$ by roughly a factor of 3; we may use the height bound $h(\Phi_\ell^{\gamma_2})\le 2\ell\log\ell+8\ell$ from \cite[Eq.\ 18]{BrokerLauterSutherland:CRTModPoly}.

Let us consider how we may modify Algorithm~1 to exploit \eqref{eq:RST}, thereby accelerating the computation of $\phi_\ell(Y)=\Phi_\ell(x,Y)\bmod q$, where $x=j(E)\in\Fq$.
Let $r(Y)=R(x,Y)\bmod q$, and similarly define $s$ and $t$ in terms of $S$ and $T$.
Rather than computing $\Phi_\ell\bmod p$ in step 5a, we compute $\Phi^{\gamma_2}_\ell\bmod p$ and derive $R$, $S$, and~$T$ from~\eqref{eq:phig2}.
We then compute polynomials $\bar{r}$, $\bar{s}$, and $\bar{t}$ mod $p$ instead of $\bar\phi$ in step~5b.
Finally, we recover $r$, $s$, and $t\bmod q$ in step 6 via the explicit CRT and output
\begin{equation}\label{eq:rst}
\phi = r^3Y^e+x(s^3-3rst)Y+x^2t^3Y^{2-e}
\end{equation}
in step 7.
Adjusting the height bound $B$ appropriately, this yields a speedup of nearly a factor of 9.  Note that we are not assuming $x=j(E)$ has a cube-root in~$\Fq$, or that $\End(E)\simeq\O$ satisfies $3\ndiv \disc(\O)$; the identity \eqref{eq:rst} holds for all $x$.

We can similarly compute $\phi_X$ and $\phi_{XX}$.
To simplify the formulas, let us define $U=(S^3-3RST)$ and $u=U(x,Y)\bmod q$.
Define $r'(Y)=(\frac{\partial}{\partial X}R)(x,Y)$ and $r''(Y)=(\frac{\partial^2}{\partial X^2}R)(x,Y)$, and similarly for $s, t$, and $u$.
Note that $u$, $u'$, and $u''$ can all be easily expressed in terms of $r, r', r'', s, s', s'', t, t'$, and $t''$.
We replace the computation of $\bar\phi_X$ and $\bar\phi_{XX}$ in step 5c with analogous computations of $\bar{r}',\bar{r}'', \bar{s}', \bar{s}'', \bar{t}'$, and $\bar{t}''$ mod $p$.  We then obtain
 $r, r', r'', s, s',  s'', t, t'$, and $t''$ via the explicit CRT mod $q$ and apply
\begin{align*}
\phi_X &= 3r^2r'Y^e + (u+x u')Y + (2xt^3+3x^2t^2 t')Y^{2-e};\\
\phi_{XX} &= (6r r'r'+3r^2r'')Y^e+(2 u'+ u'')Y\\
&\quad+ (2t^3+12xt^2 t' + 6x^2t t't' + 3x^2t^2 t'')Y^{2-e}.
\end{align*}

\subsection{Normalized isogenies}\label{subsection:normalized}
We now explain how Algorithms 1 and 2 may be used to compute normalized isogenies $\psi$, first using $j$-invariants, and then using $g$-invariants.
Throughout this section $j=j(E)\in\Fq$ denotes the $j$-invariant of a given elliptic curve $E/\Fq$, defined by $y^2=x^3+Ax+B$, and $\phi(Y)=\Phi_\ell(j,Y)$.
We use $\jt=j(\tilde{E})$ to denote a root of $\phi(Y)$ in $\Fq$. 
Our goal is to compute an equation for the image of $\psi\colon E\to\tilde{E}$, and the kernel polynomial $h_\ell(X)$ for $\psi$.

\subsubsection{Algorithm 1}\label{subsubsection:isog1}
When computing $\phi$, we also compute the optional outputs 
$\phi_X$ and $\phi_{XX}$, and then $\phi_Y(Y)=\frac{d}{dY}\phi(Y)$, $\phi_{YY}(Y)=\frac{d}{dY}\phi_Y(Y)$, and $\phi_{XY}=\frac{d}{dY}\phi_X(Y)$.
We then compute the quantities $\Phi_*(j,\jt)=\phi_*(\jt)$, for $*=X,Y,XX,XY,YY$, as
defined in \S\ref{subsection:isogenies}, and apply Elkies' algorithm \cite[Alg.\ 27]{Galbraith:MathPKC} to compute $\tilde{E}$ and $h_\ell(X)$.

\subsubsection{Algorithm 2}
Having computed $\phi$ and obtained $\jt$, we run Algorithm 2 \emph{again}, this time with the input $\jt$, obtaining $\tilde\phi(Y)=\Phi_\ell(\jt,Y)$, which we now regard as $\tilde\phi(X)=\Phi_\ell(X,\jt)$, by the symmetry of $\Phi_\ell$.
We then compute $\Phi_X(j,\jt)=(\frac{d}{dX}\tilde\phi)(j)$ and $\Phi_Y(j,\jt)=(\frac{d}{dY}\phi)(\jt)$, and the quantities
\begin{equation}\label{eq:normalized}
j'=\frac{18B}{A}j,\quad \jt' = \frac{-\Phi_X(j,\jt)}{\ell\Phi_Y(j,\jt)}j', \quad\tilde{m}=\frac{\jt'}{\jt},\quad\tilde{k}=\frac{\jt'}{1728-\jt},
\end{equation}
as in \cite[Alg.\ 27]{Galbraith:MathPKC}.
The normalized equation for $\tilde{E}$ is then $y^2 = x^3 + \frac{\ell^4\tilde{m}\tilde{k}}{48}x + \frac{\ell^6\tilde m^2\tilde k}{864}$,
and the \textsf{fastElkies}$'$ algorithm in \cite{Bostan:FastIsogenies} may be used to compute $h_\ell(X)$.

\subsubsection{Handling $g$-invariants}
We assume that $g(E)$ is known to be a class invariant (see \S\ref{subsubsection:invcheck} below).
Let $g=g(E)$, $\phi^g(Y)=\Phi_\ell^g(g,Y)$, and let $\gt=g(\tilde{E})$ denote a root of $\phi^g(Y)$ in $\Fq$.
In the case of Algorithm 1 we compute $\Phi^g_X(g,\gt)=\phi^g_X(\gt)$ and $\Phi^g_Y(g,\gt)=(\frac{d}{dY}\phi^g)(\gt)$, and in the case of Algorithm~2 we make a second call with input $\gt$ to obtain $\tilde\phi^g(X)=\Phi_\ell^g(X,\gt)$ as above.
We then compute $\Phi_X^g(g,\gt)=(\frac{d}{dX}\tilde\phi^g)(g)$ and $\Phi_Y^g(g,\gt)=(\frac{d}{dY}\phi^g)(\gt)$.
We assume the modular equation $\Psi^g_\ell(G,J)=0$ relating $g(z)$ to $j(z)$ can be solved for $j(z)$ (for the $g(z)$ considered in \cite{BrokerLauterSutherland:CRTModPoly},
 $\deg_J\Psi^g(G,J)\le 2$), and let $F(G)$ satisfy $\Psi^g_\ell(F(J),J)=0$ and $F'=\frac{d}{dG}F$.

To compute the normalized equation for $\tilde{E}$, we proceed as in \eqref{eq:normalized}, except now
\begin{equation}
\jt' = \frac{-\Phi_X^g(g,\gt)F'(\gt) }{ \ell\Phi_Y^g(g,\gt)F'(g) }j'.
\end{equation}
The \textsf{fastElkies}$'$ algorithm in \cite{Bostan:FastIsogenies} may then be used to compute $h_\ell$, or,
in the case of Algorithm~1, one may apply \cite[Alg.\ 27]{Galbraith:MathPKC} using the following identity to compute the value $r$ that appears in line 5 of \cite[Alg.\ 27]{Galbraith:MathPKC}:
\begin{align}
r &= -\frac{j'^2\Phi_{XX}(j,\jt)+2\ell j'\jt'\Phi_{XY}(j,\jt)+\ell^2\jt'^2\Phi_{YY}(j,\jt)}{j'\Phi_X(j,\jt)}\\\notag
&= -\frac{g'^2\Phi^g_{XX}(g,\gt)+2\ell g'\gt'\Phi^g_{XY}(g,\gt)+\ell^2\gt'^2\Phi^g_{YY}(g,\gt)}{j'\Phi^g_X(g,\gt)}+\frac{F''(g)}{F'(g)}g'-\ell\frac{F''(\gt)}{F'(\gt)}\gt',
\end{align}
where $g'=j'/F'(g)$ and $\gt'=\jt'/F'(\gt)$.

\subsection{Verifying that $g(E)$ is a class invariant}\label{subsubsection:invcheck}
Let $E/\Fq$ be an elliptic curve that is not supersingular (see \cite{Sutherland:Supersingular} for fast tests), with $\End(E)\simeq \O$.
As in \S\ref{subsection:modular functions}, we call an element $g(E)$ of $\Fq$ a \emph{class invariant} if (i) $H^g_\O(X)$ splits into linear factors in the ring class field of $O$, and (ii) $g(E)$ is a common root of $H^g_\O(X)$ and $\Psi^g(X,j(E))$.

For practical applications, we would like to determine whether $g(E)$ is a class invariant without computing $\O$ (indeed, the application may be to compute $\O$).  This is often easy to do, at least as far as condition (i) is concerned.
As noted in \S\ref{subsection:modular functions}, (i) can typically be guaranteed by constraints involving $D=\disc(\O)$ and the level $N$ of $g$.
Verifying condition (ii) is more difficult, in general, but it can be easily addressed in particular cases if we know that $\Psi^g(X,j(E))$ either has a unique root in $\Fq$ (which then must also be a root of $H^g(\O)$ once (i) is satisfied), or that all its roots in $\Fq$ are roots of $H^g(\O)$, or of $H^{\bar g}(\O)$ for some $\bar g$ with $\Phi_\ell^{\bar g}=\Phi_\ell^g$.
In the latter case we may not determine $g(E)$ uniquely, but for the purposes of computing a normalized $\ell$-isogeny this does not matter, any choice will work.

Taking $\gamma_2=\sqrt[3]{j}$ as an example, condition (i) holds when $\inkron{D}{3}\ne 0$, which means $j(E)$ is on the surface of its 3-volcano and has either 0 or 2 siblings.
This can be easily determined using \cite{Fouquet:IsogenyVolcanoes} or \cite[4.1]{Sutherland:HilbertClassPolynomials}.
If we have $q\equiv 2\bmod 3$, the polynomial $\Psi^g(X,j(E))=X^3-j(E)$ has a unique root $g(E)$ in $\Fq$ and (ii) also holds.\footnote{There are techniques to handle $q\equiv 1\bmod 3$, see \cite{Broker:pAdicClassInvariants} for example, but they assume $\O$ is known.}

As a second example, consider the Weber ${\frak f}$-function, which is related to the $j$-function by
$\Psi^{\mathfrak f}(X,J)=(X^{24}-16)^3-X^{24}J$.
Now we require $\inkron{D}{3}\ne 0$ and $\inkron{D}{2}=1$.
The latter is equivalent to $j(E)$ being on the surface of its 2-volcano with 2 siblings.
If we also have $q\equiv 11\bmod 12$, then $\Psi^{\mathfrak f}(X,j(E))$ has exactly two roots $\frak f(E)$ and $-\frak f(E)$, by \cite[Lemma~7.3]{BrokerLauterSutherland:CRTModPoly}, and either may be used since $\Phi_\ell^{\mathfrak f} = \Phi_\ell^{-\mathfrak f}$.

For a more general solution, having verified condition (i), we may simply compute instantiated polynomials $\phi(Y)=\Phi_\ell(x,Y)$ for \emph{every} root $x$ of $\Psi^g(X,j(E))$ in $\Fq$.  This can be done at essentially no additional cost, and we may then attempt to compute a normalized isogeny corresponding to each root $x$, which we validate by computing the dual isogeny (using the normalization factor $c=\ell$ rather than 1) and checking whether the composition corresponds to scalar multiplication by $\ell$ using randomly generated points in $E(\Fq)$.  The cost of these validations is negligible compared to the cost of computing $\phi(Y)$ for even one $x$.

As a final remark, we note that in applications such as point counting where one is only concerned with the isogeny class of $E$, in cases where condition (i) is not satisfied, one may be able to obtain an isogenous $\tilde E$ for which (i) holds by simply climbing to the surface of the relevant $\ell_0$-volcanoes for the primes $\ell_0 | N$ (we regard~$N$ as fixed so $\ell_0$ is small; $\ell_0=2,3$ in the examples above).

\section{Applications}

In this section we analyze the use of Algorithms~1 and 2 in two particular applications: point counting and computing endomorphism rings. 

Recall that for an elliptic curve $E/\Fq$, an odd prime $\ell$ is called an \emph{Elkies prime} whenever $\phi(Y) = \Phi_\ell(j(E),Y)$ has a root in~$\Fq$.  This holds if and only if $t^2 - 4q$ is a square mod $\ell$, where $t=q+1-\#E(\Fq)$.
It follows from the Chebotarev density theorem that the set of Elkies primes for $E$ has density $1/2$.
The complexity of the Schoof-Elkies-Atkin algorithm \cite{Schoof:ECPointCounting2} for computing $\#E(\Fq)$ depends critically on the number of \emph{small} Elkies primes, specifically, the least $L=L(E)$ for which
\begin{equation}\label{eq:ElkiesPrimes}
\sum_{\text{Elkies primes}\ \ell\le L(E)}\log \ell > \log (4\sqrt{q}).
\end{equation}
On average, one expects $L\approx \log q$, but even under the GRH the best proven bound is $L=O(\logpow{2+\epsilon}{q})$, see Appendix A of \cite{Satoh:padicPointCounting} by Satoh and Galbraith.  This yields a complexity bound that is actually slightly \emph{worse} than Schoof's original algorithm.

For practical purposes, the heuristic assumption $L(E)=O(\log q)$ is often used when analyzing the complexity of the SEA algorithm.
This assumption holds for almost all elliptic curves \cite{ShparlinskiSutherland:ElkiesPrimes}, but it is known to fail in infinitely many cases \cite{Shparlinski:ElkiesPrimes}.
We instead adopt the following weaker heuristic.

\begin{heuristic}\label{heuristic}
There exists a constant $c$ such that for all sufficiently large $q$ we have $L(E)\le c\log q\llog q$ for every elliptic curve $E/\Fq$.
\end{heuristic}

\begin{theorem}\label{thm:SEAbound}
Assume the GRH and Heuristic~\ref{heuristic}.
Let $E/\Fq$ be an elliptic curve over a prime field $\Fq$ and let $n=\log q$.
There is a Las Vegas algorithm to compute $\#E(\Fq)$ that runs in $O(n^4\logpow{3}{n}\llog n)$ expected time using $O(n^2\log n)$ space.
\end{theorem}
\begin{proof}
Apply the SEA algorithm, using Algorithm~1 to compute $\phi(Y)=\Phi_\ell(j(E),Y)$ (and also $\phi_X$ and $\phi_{XX}$), and ignore the Atkin primes, as in \cite[Alg.\ 1]{ShparlinskiSutherland:ElkiesPrimes}, for example.
There are $O(n/\log n)$ primes in the sum \eqref{eq:ElkiesPrimes}, and under Heuristic~\ref{heuristic}, they are bounded by $L = O(n\log n)$.
It follows from \cite[Table~1]{ShparlinskiSutherland:ElkiesPrimes} that the expected time to process each Elkies prime given $\phi$ is $O(n^3\logpow{3}{n}\llogpow{2}{n})$, which is dominated by the time to compute $\phi$, as is the space. 
The theorem then follows from Theorem~\ref{thm:alg1comp}.
\end{proof}

A common application of the SEA algorithm is to search for random curves of prime (or near prime) order, for use in cryptographic applications. As shown in~\cite{ShparlinskiSutherland:ElkiesPrimes}, we no longer need Heuristic~\ref{heuristic} to do this; we can assume $L(E)=O(\log q)$ for a randomly chosen elliptic curve.
Additionally, since we expect to count points on many curves ($\approx\log q$), we can take advantage of \emph{batching},
whereby we extend Algorithm~1 to take multiple inputs $j(E_1)\in\F_{q_1},\ldots,j(E_k)\in\F_{q_k}$ and produce corresponding outputs for each (the $\F_{q_i}$ may coincide, but they need not).
Provided $k=O(\log \ell)$, this does not change the time complexity (relative to the largest $\F_{q_i}$), since the most time-consuming steps depend only on $\ell$, not $j(E)$,
and the space complexity is increased by at most a factor of $k$.\footnote{These remarks also apply to Algorithm~2.}

Let $E_{a,b}$ denote the elliptic curve defined by $y^2=x^3+ax+b$, and for any real number $x > 3$, let $T(x)$ denote the set of all triples $(q,a,b)$ with $q\in[x,2x]$ prime, $a,b\in\Fq$, and $\#E_{a,b}$ prime.
The following result strengthens \cite[Thm.\ 3]{ShparlinskiSutherland:ElkiesPrimes}

\begin{theorem}
There is a Las Vegas algorithm that, given $x$, outputs a random triple $(q,a,b)\in T(x)$ and the prime $\#E_{a,b}(\Fq)$, with $q$ uniformly distributed over the primes in $[x,2x]$ and $(a,b)$ uniformly distributed over the pairs $(c,d)\in\Fq^2$ for which $\#E_{c,d}(\Fq)$ is prime.
Under the GRH, its expected running time is $O(n^5\logpow{2}{n}\llog n)$ using $O(n^2\logpow{2}{n})$ space, where $n=\log x$.
\end{theorem}
\begin{proof}
We modify the algorithm in \cite{ShparlinskiSutherland:ElkiesPrimes} to use Algorithm~1, operating on batches of $O(\log n)$ inputs at a time.
One then obtains an $O(n^4\log n\llog n)$ bound on the average time to compute $\#E_{a,b}(\Fq)$ for primes $q\in[x,2x]$, and a space complexity of $O(n^2\logpow{2}{n})$.
The theorem then follows from the proof of \cite[Thm.\ 3]{ShparlinskiSutherland:ElkiesPrimes}.
\end{proof}

A second application of Algorithms~1 and~2 is in the computation of the endomorphism ring of an ordinary elliptic curve.
The algorithm in \cite{BissonSutherland:Endomorphism} achieves a heuristically subexponential running time of $L[1/2,\sqrt{3}/2]$ using $L[1/2,1/\sqrt{3}]$ space.
Both Algorithms~1 and~2 improve the space complexity bound to $L[1/2,1/\sqrt{12}]$, which is significant, since space is the limiting factor in these computations.
Algorithm~2 also provides a slight improvement to the time complexity that is not visible in the $L[\alpha,c]$ notation but may be practically useful.
These remarks also apply to the algorithm in \cite{Jao:LargeIsogenies} for evaluating isogenies of large degree.

\section{Computations}\label{section:computations}

Using a modified version of the SEA algorithm incorporating Algorithm~1, we counted the number of points on the elliptic curve
\[
 y^2 = x^3 + 2718281828x + 3141592653,
\]
modulo the 5011-digit prime $q=16219299585\cdot 2^{16612} - 1$.
The algorithm ignored the Atkin primes and computed the trace of Frobenius $t$ modulo 700 Elkies primes, the largest of which was $\ell=11681$; see \cite{Sutherland:SEArecords} for details, including the exact value of~$t$, which is too large to print here.
The computation was distributed over 32 cores (3.0 GHz AMD Phenom II), and took about 6 weeks.

\begin{table}
\begin{center}
\begin{tabular}{lr}
Task & CPU time {\small (3.0 GHz AMD)}\\\hline
Compute $\phi^{\mathfrak f}_\ell(Y)$ with Algorithm~1 & 32 days\\
Compute $X^q$ mod $\phi_\ell$ (using \cite{Harvey:CacheFriendly})& 995 days\\
Compute $h_\ell$ using \cite[Alg.\ 27]{Galbraith:MathPKC}& 3 days\\
Compute $Y^q$ and $X^q$ mod $h_\ell, E$ using \cite{Gaudry:SEAeigenvalue} & 326 days\\\vspace{2pt}
Compute the eigenvalue $\lambda_\ell$ using BSGS & 22 days\\\hline
& 1378 days
\end{tabular}
\smallskip

\caption{Breakdown of time spent computing $\#E(\Fq)$}\label{table:record}
\end{center}
\end{table}

For $\ell=11681$, the size of $\phi_\ell^{\mathfrak f}(Y)=\Phi_\ell^{\mathfrak{f}}(\mathfrak{f}(E),Y)$ was under 20MB and took about two hours to compute on a single core.
As can be seen in Table~\ref{table:record}, the computation of $\phi_\ell^\mathfrak f$ accounted for less than $3\%$ of the total running time, despite being the asymptotically dominant step.
This is primarily due to the use of the Weber $\mathfrak{f}$-invariant; with a less advantageous invariant (in the worst case, the $j$-invariant with the optimization of \S\ref{subsection:fasterj}), the time spent computing $\phi_\ell$ would have been comparable to or greater than the time spent on the remaining steps.
But in any case the computation would still have been quite feasible.

To demonstrate the scalability of the algorithm, we computed $\phi_\ell^{\mathfrak f}(Y)$ for an elliptic curve $E/\Fq$, with $\ell=100019$ and $q=2^{86243}-1$.
Running on 32 cores (Algorithms~1 and 2 are both easily parallelized), this computation took less than a week.
We note that the size of the instantiated modular polynomial $\phi_\ell^{\mathfrak f}$ (and $\phi_\ell$) is almost exactly one gigabyte, whereas the size of $\Phi_\ell^\mathfrak{f}$ is many terabytes, and we estimate that the size of $\Phi_\ell$ is around 20 or 30 petabytes.

\bibliographystyle{amsplain}

\end{document}